\documentclass{article}

\usepackage{amsmath}
\usepackage{amsthm}
\usepackage{amssymb}
\usepackage{delarray}
\usepackage[dvips]{graphics}
\usepackage{epsfig}
\usepackage{color}

\newtheorem{thm}{Theorem}
\newtheorem{lemma}[thm]{Lemma}

\theoremstyle{definition}

\theoremstyle{remark}

\title{A note on the incidence coloring of outerplanar graphs.}
\author{Maksim Maydanskiy\\ MIT \\
\texttt{maksimm@math.mit.edu}}

\begin{document}

\maketitle

\begin{abstract}
 In this note we prove that every outerplanar graph is $\Delta+2$ colorable. This is slightly stronger then an unpublished result of Wang Shudong, Ma Fangfang, Xu Jin, and Yan Lijun proving the same for 2-connected outerplanar graphs.
\end{abstract}

\section{Definitions and notations.}

A graph is \emph{outerplanar} if it can be embedded in the plane without crossing edges, in such a way that all the vertices are on the boundary of the exterior region.

An \emph{incidence} of a simple graph $G$ is a pair $(v, vw)$ of an edge $vw$ and one of its vertices. Two incidences $(v,vw)$ and $(\hat{v}, \hat{v}\hat{w})$ are \emph{adjacent} if $v=\hat{v}$, or $w=\hat{v}$ or $v=\hat{w}$. 

Following Wang, Ma, Xu, and Yan, we define \emph{$(k,l)$-incidence colorings} to be a proper colorings of incidences of a given graph $G$ with at most $k$ colors such that for any vertex $v$ of $G$ the number of colors used in coloring all incidences $(u, uv)$ is at most $l$.  This notion also appears in \cite{H}.

The maximum degree of a vertex in $G$ is denoted by $\Delta$.

Finaly, the neighbourhood $N(v)$ of a vertex $v$ is a set of all vertices adjacent to $v$ in $G$.

\section{The proof.}

\begin{thm}
Any outerplanar graph $G$ has a $(\Delta+2,2)$-incidence coloring.
\end{thm}

\begin{proof}
It suffices to prove the theorem for connected graphs. We will need the following lemma.

\begin{lemma}
For every connected simple outerplanar graph $G$ at least one of the following holds:

Case 1: $G$ has a vertex of dergree 1.

Case 2: $G$ has two adjacent vertices of degree 2.

Case 3: $G$ has a vertex $u$ of degree 2 with $N(u)=(v,w)$ and $vw \in G$.

Case 4: $G$ contains a vertex $u$ of degre 2 with $G-u$ disconnected.  
 \end{lemma}
 
 The proof is based on the proof of Proposition 7.1.15 in (\cite{W}, p.254). 
 
 \begin{proof}

  Suppose $G$ has no vertex of degree 1.
 
 The following procedure exibits $G$ as a subgrpah of an outerplanar graph $H$ such that the bounadry of the unbounded face of $H$ is a cycle, i.e. a 2-connected outerplanar graph:

  If boundary of $G$ is not a cycle then it is a walk that visits some vertex $u$ twice. If $\ldots, v,u,w \ldots$ is such a visit we add the edge $vw$. We continue in this way untill we get to $H$.

Now the weak dual of $H$ is a tree and its leaves correspond to faces with exactly one internal edge. Take one such face $F$ with the internal edge $e=ab$. 

Case A) There are at least 4 edges in the boundary of $F$. Then there are 2 adjacent vertices $u,v$ on the boundary of $F$ different from $a, b$. Both of these are of degree 2 in $H$, so of degree at most $2$ in $G$. Since $G$ is connected and has no degree 1 vertices, they are both of degree 2. This is Case 2 of the lemma.

Case B) There are 3 edges in the boundary of $F$. Denote the vertex not on the edge $e$ by $u$. Again $u$ is of dgree 2 in $H$, hence also in $G$. If $e$ is in $G$ we are in Case 3 of the lemma.

If $e$ is not in $G$ then it was added in passing from $G$ to $H$, wich means that $v$ was traversed twice in the walk of the unbounded face of $G$. Then $G-u$ is disconnected, we are in Case 4.   
 
 \end{proof}

 We shall now prove the theorem by induction on order of $G$. If $\Delta =2$ it is obvious, so we assume $\Delta \geq 3$. Note that the case $\Delta=3$ follows from \cite{M}, but the resulting simplification in the proof is minor, and we prefer to keep the argument self-contained.
 We now have four cases, corresponding to the cases in the lemma:
 
 Case 1: The graph $G$ has a vertex $u$ of degree $1$. Let's denote the vertex adjacent to $u$ by $v$. Then $G^*= G - u$ is an outerplanar graph of smaller order and maximum degree at most $\Delta$. Hence by induction hypothesis $G^*$ can be $(\Delta+2,2)$-incidence colored by a coloring $\sigma^*$. We extend it to a coloring $\sigma$ of $G$. The degree of $v$ in $G^*$ is at most $\Delta-1$, so there are at most $\Delta-1$ colors used by incidences $(v, vw)$ outgoing from $v$, and at most $2$ used by the incidences $(w, wv)$ incoming into $v$. Hence there is at least one color left to color $(v, vu)$. The incidence $(u,uv)$ can be colored by one of the colors incoming into $v$. 
 
 Case 2: The graph $G$ has two adjacent vertices $u, v$ of degree $2$. Denote the other vertex adjacent to $u$ by $w$, the one adjacent to $v$ by $x$. Consider $G^* = G-u$. Again, $G$ is outerplanar, has smaller order and maximum degree at most $\Delta$ and so can be $(\Delta+2,2)$-incidence colored by a coloring $\sigma^*$. 
 
 Degree of $w$ in $G^*$ is at most $\Delta-1$, so there is at least one color $\alpha$ available to color $(w, wu)$. One of the incoming colors of $w$ can be used to color $(u, uw)$. Now we need to color $(u,uv)$ and $(v, vu)$. There are at most 4 prohibited colors and at least 5 available (as $\Delta \geq 3)$. If the color of $(w,wu)$ or $(u ,uw)$ is the same as the color of $(x, xv)$ then there are at most 3 prohibited colors, and we can use 2 remaining ones to finish the coloring. If all $(w,wu)$, $(u ,uw)$ and $(x, xv)$ have distinct colors, we can use the color of $(x, xv)$ to color $(u, uv)$, and have a color left to finish coloring $(v,vu)$.  Resulting coloring is in fact a $(\Delta+2, 2)$ coloring.
 
Case 3: The graph $G$ has a vertex $u$ of degree $2$ with $N(u)=(v,w)$ and $vw \in G$. Consider $G^* = G-u$. Again, $G$ is outerplanar, has smaller order and maximum degree at most $\Delta$ and so can be $(\Delta+2,2)$-incidence colored by a coloring $\sigma^*$. Suppose $(v, vw)$ is colored by color $\alpha$ and $(w, wv)$ by color $\beta$. 

We now assign color $\alpha$ to $(u, uw)$ and color $\beta$ to $(u, uv)$. This does not produce any conflicts since $\alpha$ already was an incoming color for $w$ and $\beta$ for $v$, and $\alpha \neq \beta$. Finally, the vertex $v$ has degree at most $\Delta -1$ in $G^*$ so there is at least one color $\gamma$, $\gamma \neq \alpha, \beta$, that can be used to color $(v, vu)$. Similarly, there is a color $\delta$, $\delta \neq \alpha, \beta$, that can be used to color $(w,wu)$ (it is possible that $\delta=\gamma)$. This produces a $(\Delta+2, 2)$- incidence coloring of $G$.
 
Case 4: The graph $G$ has a vertex $u$ of degree $2$ such that $G-u$ is disconnected. 

Again $G^*$ is outerplanar, of smaller order and maximal degree at most $\Delta$, hence $(\Delta+2, 2)$ colorable.

Denote $N(u)=(v,w)$. Let a $(\Delta+2, 2)$ coloring of the component of $G^*$ containing $v$ be $\sigma_1$ and a $(\Delta+2, 2)$ coloring of the component of $G^*$ containing $w$ be $\sigma_2$. Other components of $G-u$ are components of $G$, they can be $(\Delta+2,2)$ colored and left unmodified. Since degrees of $v$ and $w$ are at most $\Delta-1$ there exists a way to color incidences $(v,vu)$ by $\alpha$ and $(w, wu)$ by $\beta$, and then to assign one of the clors $\gamma$ incoming to $v$ to the incidence $(u, uv)$ and  one of the clors $\delta$ incoming to $w$ to the incidence $(u, uw)$.
 The problem is that while $\alpha \neq \gamma$ and $\beta \neq \delta$ there may be other equalities, so we get adjacent incidences at $u$ colored in the same way. However, the set of colors has at least 4 elements. Hence for any colors $\beta, \delta$ exists a permutation of colors sending $\beta, \delta$ to colors different from $\alpha, \gamma$. Composing $\sigma_2$ (together with the coloings of $(w, wu)$ and $(u, uw)$) with this permutation gives a $(\Delta+2, 2)$ coloring of $G$.

This completes the proof.
\end{proof}

\section{Questions on the incidence coloring of planar and higher-genus graphs.}
Even though not every graph is $(\Delta+2)$-colorable (c.f. \cite{G}), the counterexamples known to me are not planar.  The question of whether planar graphs are $(\Delta+2)$-colorable is unsolved. The bound of $\Delta+7$ was obtained in  \cite{H}. More generally, in the same paper it is shown that any k-degenerate graph has a $(\Delta+2k-1, k)$ incidence coloring. Any graph of positive genus $g$ has a vertex of degree  at most $d= \frac{1}{2} (7+\sqrt(1+48g)$, and hence is $d$-degenerate, producing a bound of $\Delta+6+\sqrt(1+48g)$ on the incidence coloring number. Planar graphs are 5-degenerate, and outerplanar graphs are 2-degenerate, so the resulting bounds of $\Delta+9$ and $\Delta+3$, respectively, are not optimal.  The higher-genus bounds are probably not tight either.

\end{document}